\documentclass[11pt,a4paper,reqno]{amsart}

\usepackage{amsmath}
\usepackage{amsfonts}
\usepackage{a4wide}
\usepackage{amssymb}
\usepackage{amsthm}

\theoremstyle{plain}

\newtheorem{teo}{Theorem}[section]
\newtheorem{lemma}[teo]{Lemma}
\newtheorem{prop}[teo]{Proposition}

\newtheorem{ackn}{Acknowledgments\!}

\theoremstyle{definition}

\theoremstyle{remark}

\newtheorem{rem}[teo]{Remark}

\numberwithin{equation}{section}

\def\SS{{{\mathbb S}}}

\def\RR{{\mathbb R}}

\def\RRR{{\mathrm R}}
\def\PPP{{\mathrm P}}

\def\WWW{{\mathrm W}}

\def\Ric{{\mathrm {Ric}}}

\title[Complete gradient shrinking Ricci solitons with pinched curvature]{Complete gradient shrinking Ricci solitons\\ with pinched curvature}
\date{\today}

\author{Giovanni Catino}
\address[Giovanni Catino]{SISSA -- International School for Advanced Studies, \mbox{Via Bonomea 265,   34136 Trieste}}
\email{catino@sissa.it}

\date{\today}

\begin{document}

\begin{abstract} We prove that any $n$--dimensional complete gradient Ricci soliton with pinched Weyl curvature is a finite quotient of $\RR^{n}$, $\RR \times \SS^{n-1}$ or $\SS^{n}$. In particular, we do not need to assume the metric to be locally conformally flat.
\end{abstract}

\maketitle

\section{Introduction}

In this paper we classify complete gradient shrinking Ricci solitons satisfying a pointwise pinching condition. We recall that a complete Riemannian manifold $(M^n ,g)$ is a {\em gradient Ricci solitons} if there exists a smooth function $f$ on $M^n$ such that 
$$
\Ric + \nabla^2 f \, = \, \lambda g
$$ 
for some constant $\lambda$. The Ricci soliton is called {\em shrinking} if $\lambda>0$, {\em steady} if $\lambda=0$ and {\em expanding} if $\lambda<0$. Ricci solitons play a fundamental role in the formation of singularities of the Ricci flow, and have been studied by many authors (see H.-D. Cao~\cite{cao2} for a nice overview). 

\medskip 

In recent years much attention has been given to the classification of complete gradient shrinking solitons, in particular in dimensions $n = 3$ (Ivey~\cite{ivey1} for the compact case and Perelman~\cite{perel1}, Ni--Wallach~\cite{nw2} and Cao--Chen--Zhu~\cite{caochenzhu} for the complete case), in dimension $n = 4$ (Ni--Wallach~\cite{nw1} and Naber~\cite{na1}), as well as in the locally conformal flat case when $n \geq 4$ (Ni--Wallach~\cite{nw2}, Petersen--Wylie~\cite{pw3} and Z.-H. Zhang~\cite{zhang}). In all three cases, the proofs depend on the crucial fact that the shrinking soliton under the consideration has nonnegative curvature operator. Notice that this condition is automatically satisfied when $n = 3$ (B.-L. Chen~\cite{chen2}) or when the shrinker is locally conformally flat (Z.-H. Zhang~\cite{zhang}). However, when $n \geq 4$ one cannot expect that in general a complete gradient shrinking soliton has nonnegative curvature operator (or even nonnegative Ricci curvature).

\medskip

In this paper we will generalize the previous results concerning the classification of complete gradient shrinking Ricci solitons to the case when the Ricci tensor is nonnegative and a very general pinching condition on the Weyl tensor is in force. In particular we will not assume the soliton metric to be  locally conformally flat. 

\medskip

To fix the notation we recall that the Riemann curvature operator of a Riemannian manifold $(M^n,g)$ is defined as in~\cite{gahula} by
$$
\mathrm{Riem}(X,Y)Z=\nabla_{Y}\nabla_{X}Z-\nabla_{X}\nabla_{Y}Z+\nabla_{[X,Y]}Z\,.
$$ 
In a local coordinate system the components of the $(3,1)$--Riemann curvature tensor are given by
$\RRR^{l}_{ijk}\tfrac{\partial}{\partial
  x^{l}}=\mathrm{Riem}\big(\tfrac{\partial}{\partial
  x^{i}},\tfrac{\partial}{\partial
  x^{j}}\big)\tfrac{\partial}{\partial x^{k}}$ and we denote by
$\RRR_{ijkl}=g_{lp}\RRR^{p}_{ijk}$ its $(4,0)$--version.

\medskip

{\em In all the paper the Einstein convention of summing over the repeated indices will be adopted.}

\medskip

The Ricci tensor $\Ric$ is obtained by the contraction 
$\RRR_{ik}=g^{jl}\RRR_{ijkl}$, $\RRR=g^{ik}\RRR_{ik}$ will denote the scalar curvature and $\overset{\circ}{\Ric}=\Ric-\tfrac{1}{n}\RRR\, g$ the traceless Ricci tensor. The so called Weyl tensor $\WWW$ is then defined by the following decomposition formula (see~\cite[Chapter~3,
Section~K]{gahula}) in dimension $n\geq 3$,
\begin{eqnarray*}
\WWW_{abcd}=&\,\RRR_{abcd}+\frac{\RRR}{(n-1)(n-2)}(g_{ac}g_{bd}-g_{ad}g_{bc})
- \frac{1}{n-2}(\RRR_{ac}g_{bd}-\RRR_{ad}g_{bc}
+\RRR_{bd}g_{ac}-\RRR_{bc}g_{ad})\,.
\end{eqnarray*} 

\medskip 

Now we can state our result

\begin{teo}\label{mainteo} Any $n$--dimensional complete gradient shrinking Ricci soliton with nonnegative Ricci curvature and satisfying
\begin{equation}\label{pinching}
|\WWW| \, \RRR \,  \leq \, \sqrt{\frac{2(n-1)}{n-2}} \left(\, |\overset{\circ}{\Ric}| - \tfrac{1}{\sqrt{n(n-1)}} \RRR \, \right)^{2}
\end{equation}
is a finite quotient of $\RR^{n}$, $\RR \times \SS^{n-1}$ or $\SS^{n}$. 
\end{teo}

\begin{rem}\label{rem1} The pinching~\eqref{pinching} is trivially satisfied in dimension three, since the Weyl tensor vanishes, whereas if $n\geq 4$, it generalizes all the previous results concerning locally conformally flat gradient shrinkers. Moreover this condition does not imply a priori the nonnegativity of the curvature operator.
\end{rem}

\begin{rem}\label{rem2} As we will see in the proof,  the condition on the Ricci curvature can be relaxed to an estimate of the type $|\Ric|\leq c\, \RRR^{1+\alpha}$, for some constants $c>0$ and $\alpha\geq 0$. In particular we do not need to assume a priori any kind of positivity on the Ricci curvature.
\end{rem}

\bigskip

\section{Proof of Theorem~\ref{mainteo}} 

We recall the following formulas (for the proof see Eminenti--La Nave--Mantegazza~\cite{mantemin2}) which will be useful in the rest of this section

\begin{lemma}\label{formulas} Let $(M^{n},g)$ be a gradient Ricci soliton, then the following formulas hold
\begin{equation}\label{eq1}
\Delta\RRR \,= \,\langle\nabla\RRR , \nabla
f\rangle+2\,\lambda\,\RRR-2|\Ric|^2
\end{equation}
\begin{align}
\Delta\RRR_{ik}
\,=\,&\,\langle\nabla\RRR_{ik}, \nabla
f\rangle+2\,\lambda\,\RRR_{ik}-2\,\WWW_{ijkl}\RRR^{jl}\label{eq2}\\
&\,+\tfrac{2}{(n-1)(n-2)}
\bigl(\RRR^2 g_{ik}-n\RRR\,\RRR_{ik}
+2(n-1)\RRR_{ij}\RRR^{j}_{\,k}-(n-1)|\Ric|^{2} g_{ik}\bigr)\,.\nonumber
\end{align}
\end{lemma}

\medskip

It is well known that a complete gradient shrinking Ricci soliton has nonnegative scalar curvature (see B.-L. Chen~\cite{chen2}). On the other hand, due to the geometric properties of gradient shrinkers, we know that either $\RRR>0$ or the metric $g$ is flat (see Pigola--Rimoldi--Setti~\cite[Theorem 3]{pigrimset}). 

\medskip

Hence, from now on we can assume the scalar curvature of $g$ to be strictly positive.

\begin{prop}\label{propstima} Let $(M^{n},g)$ be a complete non--flat gradient shrinking Ricci soliton. Then the following estimate holds
\begin{align*}
\Delta \left(\frac{|\Ric|^{2}}{\RRR^{2}}\right) \, \geq \, & \, \langle \nabla \left(\frac{|\Ric|^{2}}{\RRR^{2}}\right), \nabla f - \nabla \log \RRR^{2} \rangle + \frac{2}{R^{4}}\,|\,\RRR \,\nabla_{j} \RRR_{ik} - \RRR_{ik}\,\nabla_{j} \RRR\,|^{2} \\
&\, +\, \frac{4}{\RRR^{3}}\left[ \,|\overset{\circ}{\Ric}|^{2} \left( |\overset{\circ}{\Ric}| - \tfrac{1}{\sqrt{n(n-1)}}\RRR\right)^{2} - \RRR \,\WWW_{ijkl}\hbox{$\overset{\circ}{\RRR}$}\,\hspace{-0.05cm}^{ik} \overset{\circ}{\RRR}\,\hspace{-0.05cm}^{jl}\,\right] \,.
\end{align*}
\end{prop}
\begin{proof}
From Lemma~\ref{formulas}, one can easily compute
$$
\Delta \RRR^{2} \, = \, \langle \nabla \RRR^{2}, \nabla f \rangle + 2 |\nabla \RRR|^{2} + 4\,\lambda\,\RRR^{2} - 4\RRR|\Ric|^{2}
$$
and
\begin{align*}
\Delta|\Ric|^{2} \, = \,&\, \langle \nabla|\Ric|^{2}, \nabla f \rangle + 2 |\nabla \Ric|^{2} + 4\,\lambda\,|\Ric|^{2} - 4 \,\WWW_{ijkl}\hbox{$\overset{\circ}{\RRR}$}\,\hspace{-0.05cm}^{ik} \overset{\circ}{\RRR}\,\hspace{-0.05cm}^{jl} \\
&\,+ \tfrac{4}{(n-1)(n-2)}\left(\RRR^{3}-(2n-1)\RRR|\Ric|^{2}+2(n-1)\RRR_{ij}\RRR^{jk}\RRR^{i}_{\,k}\right) \,,
\end{align*}
where we have used the fact that $\WWW_{ijkl}\RRR^{ik}\RRR^{jl}=\WWW_{ijkl}\hbox{$\overset{\circ}{\RRR}$}\,\hspace{-0.05cm}^{ik} \overset{\circ}{\RRR}\,\hspace{-0.05cm}^{jl}$, since all the traces of the Weyl tensor vanish. Thus, we obtain

\begin{align}\label{eq7}
\Delta \left(\frac{|\Ric|^{2}}{\RRR^{2}}\right) \, = \, & \, \langle \nabla \left(\frac{|\Ric|^{2}}{\RRR^{2}}\right), \nabla f - \nabla \log \RRR^{2} \rangle + \frac{2}{R^{4}}\,|\,\RRR \,\nabla_{j} \RRR_{ik} - \RRR_{ik}\,\nabla_{j} \RRR\,|^{2} \\\nonumber
&\, +\, \frac{4}{\RRR^{3}}\left( \PPP - \RRR \,\WWW_{ijkl}\hbox{$\overset{\circ}{\RRR}$}\,\hspace{-0.05cm}^{ik} \overset{\circ}{\RRR}\,\hspace{-0.05cm}^{jl}\,\right)\,
\end{align}
where 
$$
\PPP \, = \, \tfrac{1}{(n-1)(n-2)}\left(\,\RRR^{4}-(2n-1)\RRR^{2}|\Ric|^{2}+2(n-1)\RRR\,\RRR_{ij}\RRR^{jk}\RRR^{i}_{\,k} + (n-1)(n-2)|\Ric|^{4} \,\right)\,.
$$
We will show that 
$$
\PPP \, \geq \, |\overset{\circ}{\Ric}|^{2} \left( |\overset{\circ}{\Ric}| - \tfrac{1}{\sqrt{n(n-1)}}\RRR\right)^{2}\,
$$
with equality if and only if either $\overset{\circ}{\Ric}=0$ or $|\overset{\circ}{\Ric}| = \tfrac{1}{\sqrt{n(n-1)}}\RRR$. 

Using the fact that $\overset{\circ}{\RRR}_{ik} = \RRR_{ik} - \tfrac{1}{n}\RRR g_{ik}$, one has the formula
$$
\RRR_{ij}\RRR^{jk}\RRR^{i}_{\,k} \,=\, \overset{\circ}{\RRR}_{ij}\overset{\circ}{\RRR}\hspace{-0.05cm}\,^{jk}\overset{\circ}{\RRR}\hspace{-0.05cm}\,^{i}_{\,k} + \tfrac{3}{n}\RRR |\Ric|^{2} - \tfrac{2}{n^{2}} \RRR^{3} \,.
$$
By applying the well known estimate (for the proof see, for instance, Huisken~\cite[Lemma 2.4]{huisk8}) 
$$
\overset{\circ}{\RRR}_{ij}\overset{\circ}{\RRR}\hspace{-0.05cm}\,^{jk}\overset{\circ}{\RRR}\hspace{-0.05cm}\,^{i}_{\,k} \,\geq \, -\tfrac{n-2}{\sqrt{n(n-1)}} |\overset{\circ}{\Ric}|^{3} \,,
$$
we obtain
\begin{equation}\label{stima1}
\RRR_{ij}\RRR^{jk}\RRR^{i}_{\,k} \,\geq\, -\tfrac{n-2}{\sqrt{n(n-1)}} |\overset{\circ}{\Ric}|^{3} + \tfrac{3}{n}\RRR |\Ric|^{2} - \tfrac{2}{n^{2}} \RRR^{3} \,,
\end{equation}
with equality if and only if either $\overset{\circ}{\Ric}=0$ or $|\overset{\circ}{\Ric}| = \tfrac{1}{\sqrt{(n(n-1)}}\RRR$. Using inequality~\eqref{stima1} and the fact that $|\Ric|^{2} = |\overset{\circ}{\Ric}|^{2}+\tfrac{1}{n} \RRR^{2}$, from the definition of $\PPP$, we get
\begin{eqnarray*}
\PPP \, &\geq& \, \tfrac{1}{(n-1)(n-2)}\left( (n-1)(n-2)|\overset{\circ}{\Ric}|^{4} -\tfrac{2(n-1)(n-2)}{\sqrt{n(n-1)}}\RRR\,|\overset{\circ}{\Ric}|^{3}+ \tfrac{n-2}{n} \RRR^{2}\, |\overset{\circ}{\Ric}|^{2} \right) \\
&=& |\overset{\circ}{\Ric}|^{2} \left( |\overset{\circ}{\Ric}| - \tfrac{1}{\sqrt{n(n-1)}}\RRR\right)^{2}\,.
\end{eqnarray*}
The proposition now follows from equality~\eqref{eq7}.

\end{proof}

\

Now, let $h=f-\log \RRR^{2}$. Following the notation in Petersen-Wylie~\cite{pw3}, from Proposition~\ref{propstima}, we get that, if $g$ is non--flat, then
\begin{align*}
\Delta_{h} \left(\frac{|\Ric|^{2}}{\RRR^{2}}\right) \, \geq \,   \frac{4}{\RRR^{3}}\left[ \,|\overset{\circ}{\Ric}|^{2} \left( |\overset{\circ}{\Ric}| - \tfrac{1}{\sqrt{n(n-1)}}\RRR\right)^{2} - \RRR \,\WWW_{ijkl}\hbox{$\overset{\circ}{\RRR}$}\,\hspace{-0.05cm}^{ik} \overset{\circ}{\RRR}\,\hspace{-0.05cm}^{jl}\,\right] \,,
\end{align*}
where $\Delta_{h}=\Delta-\nabla_{\nabla h}$. In order to estimate the term involving the Weyl curvature, instead of using the Cauchy--Schwarz inequality, we recall this refined inequality which was first proved by Huisken~\cite[Lemma 3.4]{huisk8}
\begin{lemma} We have the estimate
$$
|\WWW_{ijkl}\hbox{$\overset{\circ}{\RRR}$}\,\hspace{-0.05cm}^{ik} \overset{\circ}{\RRR}\,\hspace{-0.05cm}^{jl}| \, \leq \, \sqrt{\frac{n-2}{2(n-1)}} \,|\WWW| \,|\overset{\circ}{\Ric}|^{2} \,.
$$
\end{lemma} 
Hence, from this inequality and the pinching assumption~\eqref{pinching}, one has
$$
|\overset{\circ}{\Ric}|^{2} \left( |\overset{\circ}{\Ric}| - \tfrac{1}{\sqrt{n(n-1)}}\RRR\right)^{2} - \RRR \,\WWW_{ijkl}\hbox{$\overset{\circ}{\RRR}$}\,\hspace{-0.05cm}^{ik} \overset{\circ}{\RRR}\,\hspace{-0.05cm}^{jl} \, \geq \, 0 \,.
$$
Thus, we have proved that
$$
\Delta_{h} \left(\frac{|\Ric|^{2}}{\RRR^{2}}\right) \,\geq\,0\,
$$
where $h=f-\log\RRR^{2}$, with equality if and only if either $\overset{\circ}{\Ric}=0$ or $|\overset{\circ}{\Ric}| = \tfrac{1}{\sqrt{n(n-1)}}\RRR$. We are now in the position to apply a Liouville type theorem proved by Petersen--Wylie~\cite[Theorem 4.2]{pw3} (see also Naber~\cite{na1})

\begin{teo}\label{liouville} Let $(M^{n},g)$ be a manifold with finite $h$--volume: $\int e^{-h}dV_{g}< +\infty$. If $u$ is a smooth function in $L^{2}(e^{-h}dV_{g})$ which is bounded below such that $\Delta_{h} u \geq 0 $, then $u$ is constant.
\end{teo} 

To apply this theorem to the function $u=\frac{|\Ric|^{2}}{\RRR^{2}}$ with $h=f-\log\RRR^{2}$, we have to check that
$$
\int_{M}e^{-h}\,dV_{g} \,= \, \int_{M} \RRR^{2} \, e^{-f}\,dV_{g} \,<\, +\infty 
$$
and
$$
\int_{M}\frac{|\Ric|^{4}}{\RRR^{4}}\,e^{-h}\,dV_{g} \, = \, \int_{M} \frac{|\Ric|^{4}}{\RRR^{2}} \,e^{-f} \,dV_{g} \, < \, +\infty\,.
$$
Since, by assumption, $g$ has nonnegative Ricci curvature, we have that $|\Ric|\leq \RRR$. Hence, to estimate the above integrals it is enough to prove that $\RRR \in L^{2}(e^{-f}dV_{g})$. To this aim, we recall that H.-D. Cao and D. Zhou~\cite{caozhou} have proved that for any fixed origin $p\in M^{n}$, there exist two positive constants $c_{1}$ and $c_{2}$ so that, for every $q\in M^{n}$, we have
$$
\tfrac{1}{4}\big(r(q)-c_{1}\big)^{2} \,\leq\, f (q) \,\leq\, \tfrac{1}{4}\big(r(q)+c_{1}\big)^{2} \,,
$$
where $r(q)=\hbox{dist}(q,p)$, and
$$
\hbox{Vol}\big(B_{\rho}(p)\big) \,\leq\, c_{2} \, \rho^{n} \,,
$$
for sufficiently large $\rho$. Moreover, from the well known equation satisfied by gradient shrinking solitons (see Hamilton~\cite{hamilton9})
$$
\RRR+|\nabla f|^{2}-f=\hbox{const}
$$
we can estimate the scalar curvature with $f$ to conclude that
\begin{equation}\label{l2scalar}
\int_{M}\RRR^{2}\,e^{-f}\,dV_{g} \,<\,+\infty \,.
\end{equation}

As we have observed in Remark~\ref{rem2}, to get the integrability conditions, it is sufficient to assume $|\Ric|\leq c\, \RRR^{1+\alpha}$, for some constants $c>0$ and $\alpha\geq 0$, since on any gradient shrinking solitons, the scalar curvature $\RRR\in L^{p}(e^{-f}dV_{g})$, for every $1\leq p < \infty$.
 
\medskip
 
Thus, Theorem~\ref{liouville} implies that $\tfrac{|\Ric|^{2}}{\RRR^{2}}$ is constant, and from the proof of Proposition~\ref{propstima}, we get that $g$ is either Einstein or satisfies $|\overset{\circ}{\Ric}| = \tfrac{1}{\sqrt{n(n-1)}}\RRR$. Now, the pinching assumption~\eqref{pinching} implies that either $g$ is Einstein or has zero Weyl tensor. 

In the first case, since $\overset{\circ}{\Ric}=0$, we have that $M^{n}$ is compact. Moreover, from the pinching condition~\eqref{pinching}, we get that $|\WWW|^{2}\leq \tfrac{2}{n^{2}(n-1)(n-2)} \RRR^{2}\leq \tfrac{4}{n(n-1)(n-2)(n+1)}\RRR^{2}$, which implies that $g$ has positive curvature operator (see Huisken~\cite[Corallary 2.5]{huisk8}). Hence, from a theorem of Tachibana~\cite{tachib1} we conclude that $g$ has constant positive sectional curvature and $(M^{n},g)$ is a finite quotient of $\SS^{n}$.

On the other hand, if the Weyl tensor of $g$ vanishes, from the classification of locally conformally flat gradient shrinking solitons we obtain that if $g$ is non--flat and non--compact, then it must be a finite quotient of $\RR\times\SS^{n-1}$.

\medskip

This concludes the proof of Theorem~\ref{mainteo}.

\bigskip

\bigskip

\begin{ackn}
\noindent The author is partially supported by the Italian project FIRB--IDEAS ``Analysis and Beyond''.
\end{ackn}

\bigskip

\bigskip

\bibliographystyle{amsplain}
\bibliography{/Users/giovannicatino/Documents/Matematica/Ricerca/Bib/biblio.bib}

\bigskip

\bigskip

\parindent=0pt

\end{document}